\documentclass[11pt]{amsart}

\usepackage[top=30truemm,bottom=30truemm,left=30truemm,right=30truemm]{geometry}
\usepackage{mathptmx}      
\usepackage{enumitem}
\usepackage{amsmath}
\usepackage{amssymb}
\usepackage{amsthm}
\usepackage{amscd}
\usepackage[all]{xy}
\usepackage{bm}
\usepackage{hyperref}
\hypersetup{
setpagesize=false,
 bookmarksnumbered=true,%
 bookmarksopen=true,%
 colorlinks=true,%
 linkcolor=black,
 citecolor=blue,
}
\usepackage{graphicx}
\usepackage{mathrsfs}

\newtheorem{theorem}{Theorem}[section]

\newtheorem{proposition}[theorem]{Proposition}
\newtheorem{corollary}[theorem]{Corollary}
\newtheorem{lemma}[theorem]{Lemma}
\newtheorem{assumption}[theorem]{Assumption}

\theoremstyle{definition}
\newtheorem{definition}[theorem]{Definition}

\theoremstyle{remark}
\newtheorem{remark}[theorem]{Remark}

\numberwithin{equation}{section}

\begin{document}
\title[Bounded composition operators and stability of dynamical systems]{Bounded composition operators on functional quasi-Banach spaces and stability of dynamical systems}

\author{Isao Ishikawa}
\address{
Center for Data Science, Ehime University
\\
3, Bunkyo-cho, Matsuyama, Ehime, 790-8577, Japan.}
\email{ishikawa.isao.zx@ehime-u.ac.jp}

\subjclass[2020]{Primary 47B33, Secondary 37C05}
\keywords{Boundedness, Composition operators, Dynamical systems, Quasi-Banach spaces, Reproducing kernel Hilbert spaces, Koopman opertors}

\date{}

\begin{abstract}
In this paper, we investigate the boundedness of composition operators defined on a quasi-Banach space continuously included in the space of smooth functions on a manifold.
We prove that the boundedness of a composition operator strongly restricts the behavior of the original map, and it provides an effective method to investigate the properties of composition operators using the theory of dynamical system.
Consequently, we prove that only affine maps can induce bounded composition operators on any quasi-Banach space continuously included in the space of entire functions of one variable if the function space contains a nonconstant function.
We also prove that any polynomial automorphisms except affine transforms cannot induce bounded composition operators on a quasi-Banach space composed of entire functions in the two-dimensional complex affine space under several mild conditions.

\end{abstract}

\maketitle
\section{Introduction}
In this paper, we investigate the boundedness of composition operators defined on a quasi-Banach space continuously included in the space of smooth functions on a smooth manifold.
We connect the boundedness of composition operators and the dynamics of the original maps.
We utilize a powerful theory developed in holomorphic dynamics, in contrast to existing studies that heavily depend on the explicit structure of their function spaces.
Consequently, we prove that only affine maps can induce bounded composition operators on any quasi-Banach space continuously included in the space of entire functions on $\mathbb{C}$, in particular, on any reproducing kernel Hilbert space (RKHS) composed of entire functions on $\mathbb{C}$.

To be precise, we introduce several notions and state our main results.
We have included other basic notations used in this paper at the end of this section.
Let $X$ be a smooth manifold of dimension $d$.
Let $\mathcal{E}(X)$ be the space of smooth functions on $X$.
We equip $\mathcal{E}(X)$ with the weak Whitney topology.
A {\em quasi-Banach} space is a complete Hausdorff topological vector space $V$, with a bounded neighborhood of $0$.
We always assume that $V$ is a subspace of $\mathcal{E}(X)$ and that the inclusion $\iota \colon V \hookrightarrow \mathcal{E}(X)$ is continuous.
A typical example of $V$ is an RKHS (see Section \ref{sec: function spaces}, or \cite{SS} for more detail), Hardy spaces, and Bergman spaces (see, for example, \cite{DS04}).

Let $f\colon X \rightarrow X$ be a smooth map. 
The {\em composition operator} is a linear operator $C_f$ on $V$ defined on $\{ h \in V : h \circ f \in V\}$ such that $C_fh := h \circ f$.
There have been numerous studies on the properties of composition operators, especially where $X$ is a complex manifold.
For classical function spaces such as the Bergman space, \cite{Zh07} and \cite{CM95} are standard references.
In the case of $X = \mathbb{C}^d$, many works exist on the characterization of the boundedness of composition operators on quasi-Banach spaces. Several results show that boundedness implies the affineness of the original holomorphic map \cite{CMS, SS17, DKL17, CCG, AA17, KD20, HHK13, HBH14, IIS20, Lev12}.
These results suggest that a broad class of function spaces of entire functions have the same property.
As in Theorem \ref{mainthm2} below, we solve this problem in the one-dimensional case.
The composition operator is also important in applied mathematics, for example, in signal processing \cite{CC93, Azi02, ACCM} as the time-warping and in machine learning \cite{HIIK, IFIHK, Kawahara} as the Koopman operator.

Let 
\begin{align*}
 \mathcal{D}(\mathbb{R}^d) := \bigoplus_{n_1, \dots, n_d = 0}^\infty \mathbb{C} \partial_{x_1}^{n_1}\cdots \partial_{x_d}^{n_d}
\end{align*}
be the space of differential operators on $\mathcal{E}(\mathbb{R}^d)$ and for each $n \ge 0$, we define a finite-dimensional subspace of $\mathcal{D}(\mathbb{R}^d)$ by
\begin{align*}
 \mathcal{D}_n(\mathbb{R}^d) := \bigoplus_{n_1+ \dots + n_d \le n} \mathbb{C} \partial_{x_1}^{n_1}\cdots \partial_{x_d}^{n_d}.
\end{align*}
Let $p \in X$ and fix a local coordinate $\phi$ from an open neighborhood of $p$ onto an open subset $\mathbb{R}^d$.
Then, we define an injective linear map $\delta_{p,\phi}\colon \mathcal{D}(\mathbb{R}^d) \rightarrow \mathcal{E}(X)'$ by
\begin{align*}
\delta_{p,\phi}(D)(h):= D(h\circ \phi^{-1})(\phi(p))
\end{align*}
for $D \in \mathcal{D}(\mathbb{R}^d)$ and $h \in \mathcal{E}(X)$, where $()'$ stands for the dual space with the strong topology.
The set $\delta_{p, \phi}(\mathcal{D}(\mathbb{R}^d))$ is independent of the choice of the local coordinate $\phi$, and we define subspaces of $\mathcal{E}(X)'$ by 
\begin{align*}
  \mathcal{D}(X)_p  & := \delta_{p,\phi}(\mathcal{D}(\mathbb{R}^d)), \\
  \mathcal{D}_n(X)_p & := \delta_{p,\phi}(\mathcal{D}_n(\mathbb{R}^d)). 
\end{align*}
As we below show (Lemma \ref{lem: explicit push on a fixed point}), the dual operator $f_*:={f^*}':\mathcal{E}(X)' \rightarrow \mathcal{E}(X)'$ of $f^*:h \mapsto h\circ f$ preserves the finite  dimensional subspace $\mathcal{D}_n(X)_p$ for any $n\ge0$ if $p$ is a fixed point of $f$.
Thus, $f_*$ induces the linear map ${\rm gr}_{f_*}^n$ on $\mathcal{D}_n(X)_p/\mathcal{D}_{n-1}(X)_p$.
For $n \ge 0$, we define a surjective linear map induced by $\iota'$:
\begin{align*}
  \kappa^n_{p}\colon \mathcal{D}_n(X)_p/\mathcal{D}_{n-1}(X)_p \longrightarrow \iota'\left(\mathcal{D}_n(X)_p\right)/\iota'\left(\mathcal{D}_{n-1}(X)_p\right), 
\end{align*}
Then, we have the following theorem:
\begin{theorem}
\label{mainthm}
\label{thm: stability of dynamics}
Let $p \in X$ be a periodic point of $f$ such that $f^r(p) = p$ for some $r \ge 1 $.
Assume the following two conditions:  
\begin{enumerate}
  \item $C_f$ is a bounded linear operator on $V$, and
  \item for infinitely many $n\ge 0$, ${\rm Ker}({\kappa}^n_{p}) \subset {\rm Ker}({\rm gr}^n_{f_*^r})$. \label{condition 2}
\end{enumerate}
  Then, any eigenvalue $\alpha$ of ${\rm d}f^r_p: T_p(X) \rightarrow T_p(X)$ satisfies the inequality $|\alpha|\le 1$.
Here, $T_p(X)$ is the tangent space of $X$ at $p$.
\end{theorem}
We note that for any holomorphic $f$ and all but finitely many $p \in X$, the condition \eqref{condition 2} holds when, for example, the space $V$ is the Fock type space \cite{CMS}, $V$ is an RKHS considered in \cite{DKL17, SS17, KD20,IIS20} (we provide more general function spaces covering them in Section \ref{sec: function spaces}), and when $V$ is an infinitely dimensional quasi-Banach space and $X$ is a one-dimensional complex manifold.
The eigenvalue $\alpha$ of the Jacobian at a periodic point is a well-studied object in the theory of dynamical systems.
A periodic point violating the inequality in Theorem \ref{mainthm} is called a saddle or repelling periodic point.
Thus, we can say that any dynamics inducing a bounded composition operator on a certain quasi-Banach space $V \subset \mathcal{E}(X)$ cannot have any saddle or repelling periodic points; namely, the dynamics behaves rather tamely on $X$.

The proof is based on the existence of the ascending filtration $\{\iota'(\mathcal{D}_n(X)_p)\}_{n\ge0}$ of finite-dimensional spaces in $V'$.
It enables us to extract the information on the eigenvalues of the Jacobian via the graded module structure and the quasi-norm on $V$.

When we consider the case of $X=\mathbb{C}$, holomorphic dynamics on $\mathbb{C}$ except affine maps chaotically behave and have infinitely many repelling points according to the theory of holomorphic dynamics in one variable \cite{Bea91, Sch10}.
As a result, we have the following result:
\begin{theorem} \label{mainthm2}\label{thm: affineness in C, banach}
Let $V \subset \mathcal{A}(\mathbb{C})$ be a quasi-Banach space, and the inclusion is continuous.
Assume $V$ contains a nonconstant entire function.
Then, if a holomorphic map $f\colon \mathbb{C} \to \mathbb{C}$ induces a bounded linear operator $C_f$ on $V$, then $f(z) = az + b$ for some $a, b \in \mathbb{C}$.
In addition, we have $|a| \le 1$ if $V$ is infinite-dimensional.
\end{theorem}
As mentioned above, these sorts of results are known in various existing works, and we prove that it is always true in the one-dimensional case.

The problem becomes considerably more complex for higher-dimensional cases than the one-dimensional one. 
Actually, we can easily construct a counter-example if we impose no condition on $V$.
However, let
\[
\mathcal{G}_d(V) := \left\{ A \in {\rm GL}_d(\mathbb{C}) : C_{A(\cdot) + b}\text{ is bounded on $V$ for some $b\in\mathbb{C}^d$}\right\}.
\]
Then, we have the following result in the two-dimensional case:
\begin{theorem} \label{mainthm3}\label{thm: two dim case}
Let $V \subset \mathcal{A}(\mathbb{C}^2)$ be a quasi-Banach space, and the inclusion map is continuous.
Assume the following two conditions:
\begin{enumerate}
  \item $\kappa_p^{n, \rm hol}$ (see Section \ref{subsec: stability cpx mfd} for the definition) is injective for all but finitely many $p \in \mathbb{C}^2$ and infinitely many $n \ge 0$, and
  \item $\langle \mathcal{G}_2(V) \rangle_\mathbb{C} = {\rm M}_2(\mathbb{C})$.
\end{enumerate}
    Then, if a polynomial automorphism $f\colon \mathbb{C}^2 \to \mathbb{C}^2$ induces a bounded linear operator $C_f$ on $V$, then $f(z) = Az + b$ for some $A \in {\rm GL}_2(\mathbb{C})$ and $ b \in \mathbb{C}$.
\end{theorem}
Here, $\langle S \rangle_\mathbb{C}$ denotes the vector space generated by $S$ over $\mathbb{C}$.
The above two conditions hold in a certain large class of function spaces, and in particular, the RKHSs treated in \cite{SS17, DKL17, IIS20} satisfy them (Section \ref{sec: function spaces}); although it is worth mentioning that these studies have already proved that not only polynomial automorphisms but holomorphic maps except affine maps can also not induce a bounded composition operators.
The theory of holomorphic dynamics is effective machinery, and it allows us to reduce the problem to proving the existence of infinitely many repelling or saddle periodic points for a suitable class of holomorphic dynamical systems (see \cite[Question 2.16]{10.2307/43736735}).
For the details, see Section \ref{sec: Cd}.

The rest of this paper is structured as follows:
In Section \ref{sec: main theorem}, we introduce notions and several preliminary lemmas and provide the proof of Theorem \ref{mainthm}.
We also obtain similar results in complex manifolds and briefly summarize their statements.
In Section \ref{sec: function spaces}, we provide several explicit examples of quasi-Banach spaces continuously included in $\mathcal{E}(X)$, which satisfies the condition \eqref{condition 2} in Theorem \ref{mainthm}.
We also illustrate that our framework covers many previous results.
In Section \ref{sec: Cd}, we prove Theorems \ref{mainthm2} and \ref{mainthm3}.

\section*{Basic notation}
We denote the set of the real (resp. complex) values by $\mathbb{R}$ (resp. $\mathbb{C}$).
We denote the set of non-negative integers by $\mathbb{Z}_{\ge 0}$.
For any subset of $S \subset \mathbb{R}$, we denote by $S_{>0}$ (resp. $S_{\ge 0}$) the set of positive (resp. non-negative) elements of $S$.
We denote by $\partial_{x_j}$ (resp. $\partial_{z_j}$) the partial derivative with respect to the $j$-th variable of a differentiable (resp. holomorphic) function on an open set of $\mathbb{R}^d$ (resp. $\mathbb{C}^d$).
We denote by ${\rm GL}_d(\mathbb{C})$ (resp. ${\rm M}_d(\mathbb{C})$) the set of regular matrices (resp. matrices) of size $d$.
For any subset $S\subset V$ of a complex linear space $V$, we denote by $\langle S \rangle_\mathbb{C}$ the linear subspace generated by $S$.

We sometimes use the multi-index notation.
For example, for $(\alpha_1,\dots, \alpha_d) \in \mathbb{Z}_{\ge 0}$, $(z_1, \dots, z_d)$ and $(\partial_{x_1},\dots, \partial_{x_d})$, we write $z^\alpha := z_1^{\alpha_1}\dots z_d^{\alpha_d}$ and $\partial_{x}^\alpha := \partial_{x_1}^{\alpha_1}\dots\partial_{x_d}^{\alpha_d}$.

For a smooth (complex) manifold $X$, we denote by $\mathcal{E}(X)$ (resp. $\mathcal{A}(X)$) the space of $\mathbb{C}$-valued $C^\infty$-class (resp. holomorphic) functions on $X$.

For a topological linear space $V$, we denote the strong dual by $V'$.
For a bounded linear map $L: V_1 \rightarrow V_2$ between topological linear spaces $V_1$ and $V_2$, we denote by $L'$ the dual linear operator $L': V_2' \rightarrow V_1'$.

For a vector space $V$ with an ascending filtration, which is an ascending family of subspaces $\{V_n \}_{n\ge -1}$, we define ${\rm gr}^n[V]:= V_n/V_{n-1}$ and ${\rm gr}[V] := \oplus_{n \ge 0} {\rm gr}^n[V]$.
For a linear map $A: V \rightarrow W$ such that there exists families of ascending subspaces $\{V_n\}_{n\ge -1}$ (resp. $\{W_n\}_{n\ge -1}$) of $V$ (resp. $W$) satisfying $A(V_n) \subset W_n$, we denote by ${\rm gr}^n_A\colon V_n/V_{n-1} \rightarrow W_n/W_{n-1}$ the induced linear map.
We define ${\rm gr}_A:= \oplus_n {\rm gr}^n_A\colon \oplus_{n\ge 0} V_n/V_{n-1} \rightarrow \oplus_{n \ge 0} W_n/W_{n-1}$

\section{Boundedness of composition operators and the stability of dynamical systems} \label{sec: main theorem}
In this section, we prove that the boundedness of a composition operator strongly limits the behavior of a dynamical system on a smooth manifold.
In Subsection \ref{subsec: stability smooth mfd}, we describe the main result in a general smooth manifold, and in Subsection \ref{subsec: stability cpx mfd}, we summarize corresponding results in the complex case.

\subsection{A stability of dynamical systems on a smooth manifold with bounded composition operators} \label{subsec: stability smooth mfd}
Let $X$ be a smooth manifold of dimension $d$ and $f\colon X \rightarrow X$ be a smooth map.
Let $V$ be a quasi-Banach space, a complete Hausdorff topological vector space with respect to a quasi-norm $\|\cdot\|_V$.
Here, the quasi-norm is a map $\|\cdot\|_V\colon V \rightarrow \mathbb{R}_{\ge 0}$ satisfying the following three conditions: (1) $\|av\|_V = |a|\cdot \|v\|_V$, (2) there exists $K \ge 1$ such that $\|v + w\|_V \le K(\|v\|_V + \|w\|_V)$, and (3) $v=0$ if $\|v\|_V=0$, where $a \in \mathbb{C}$ and $v, w \in V$.
Although a quasi-Banach space is not necessarily locally convex, bounded linear operators on a quasi-Banach space share several properties of those well-known in a Banach space.
A quasi-Banach space is also characterized as a complete Hausdorff topological vector space with a bounded neighborhood of $0$ (See \cite{kalton_peck_roberts_1984}).
In this paper, we always assume that $V \subset \mathcal{E}(X)$ and that the inclusion $V \hookrightarrow \mathcal{E}(X)$ is a continuous linear map.
A typical example is an RKHS.
We include some details of RKHS and other examples in Section \ref{sec: function spaces}.

For a smooth map $f: X \to X$, the associated composition operator on $V$ is the linear operator defined as $C_f: \{ h \in V : h \circ f \in V \} \to V$, $C_f[h]:=h \circ f$.  
This paper assumes that the domain of $C_f$ satisfies $\{h \in V : h \circ f \in V \}=V$.
Although $C_f$ is automatically bounded under this assumption by the Closed Graph Theorem as $C_f$ has a closed graph, we usually specify that $C_f$ is a bounded linear operator on $V$ in this paper to avoid confusion.

Let 
\begin{align*}
 \mathcal{D}(\mathbb{R}^d) := \bigoplus_{n_1, \dots, n_d = 0}^\infty \mathbb{C} \partial_{x_1}^{n_1}\cdots \partial_{x_d}^{n_d}
\end{align*}
be the space of differential operators on $\mathcal{E}(\mathbb{R}^d)$ and for each $n \ge 0$, we define
\begin{align*}
 \mathcal{D}_n(\mathbb{R}^d) := \bigoplus_{n_1+ \dots + n_d \le n} \mathbb{C} \partial_{x_1}^{n_1}\cdots \partial_{x_d}^{n_d}.
\end{align*}

We consider $\mathcal{E}(X)$ as a locally convex space via the following topology: 
Let $\{(U_i, \phi_i) \}_{i\in I}$ be a local coordinate system of $X$, namely $\{U_i\}_{i \in I}$ is an open covering of $X$ and $\phi_i$ is a diffeomorphism from $U_i$ into an open set of $\mathbb{R}^d$.
We take another local coordinate system $\{(V_j, \psi_j)\}_{j\in J}$ such that $\{V_j\}_{j \in J}$ is a refinement of $\{U_i \cap U_{i'}\}_{i,i' \in I}$, namely for any $j \in J$ there exists $i_j, i'_j \in I$ such that $V_j \subset U_{i_j} \cap U_{i_j'}$.
Then, we regard $\mathcal{E}(X)$ as a closed subspace of $\prod_{i \in I}\mathcal{E}(U_i)$ defined by the kernel of the following continuous linear map:
\begin{align}
\prod_{i \in I} \mathcal{E}(U_i)\longrightarrow \prod_{j \in J} \mathcal{E}(V_j); ~~(h_i)_{i\in I} \mapsto (h_{i_j}|_{V_j} - h_{i'_j}|_{V_j})_{i,j \in I}, \label{sheaf representation}
\end{align}
where we identify $\mathcal{E}(U_i)$ and $\mathcal{E}(V_j)$ with $\mathcal{E}(\phi_i(U_i))$ and $\mathcal{E}(\psi_j(V_j))$, respectively, and their topologies are defined by uniform convergence of any derivatives on compact sets (see \cite[p.26-27]{YoshidaFA}).
We equip $\mathcal{E}(X)$ with the relative topology.
As in the following lemma, this topology is independent of the choice of local coordinate systems:

\begin{lemma} \label{lem: faa di bruno}
Let $p \in \mathbb{R}^d$ and $U \subset \mathbb{R}^d$ be an open neighborhood of $p$.
Let $F: U \rightarrow \mathbb{R}^r$ be a smooth map and $h:U \rightarrow \mathbb{C}$ be a smooth function.
Then, for any $\partial_{x_{i_1}} \cdots \partial_{x_{i_n}} \in \mathcal{D}_n(\mathbb{R}^d)$ ($i_j \in \{1,\dots,d\}$), there exists a smooth map $\mathbf{D}_n$ from $U$ to $\mathcal{D}_{n-1}(\mathbb{R}^d)$ such that for all $p \in U$,
\begin{align}
  (\partial_{x_{i_1}} \cdots\partial_{x_{i_n}}) (h\circ F)(p) = \left[\bigg( \prod_{j=1}^n \partial_{x_{i_j}} F_p \cdot \bm{\partial}\bigg)h\right] (F(p)) + (\mathbf{D}_n(p)h)(F(p)).
\label{to prove formula dayo}
\end{align}
Here, we regard $\mathcal{D}_{n-1}(\mathbb{R}^d)$ as a finite product of $\mathbb{C}$'s as it is finite-dimensional, and 
\[ \partial_{x_j}F_p \cdot \bm{\partial} := \sum_{m=1}^r \frac{\partial F^m}{\partial x_j}(p)\partial_{x_m} \in \mathcal{D}_1(\mathbb{R}^d),\]
where we write $F=(F^1,\dots, F^r)$ for some $F^j \in \mathcal{E}(U)$ ($j=1,\dots,r$).
\end{lemma}
\begin{proof}
Although it can be shown by direct computation using the multivariate version of the Fa\`a di Bruno formula (see \cite[Theorem 2.1]{10.2307/2155187}), we here provide the proof using induction:

In the case of $n=1$, it immediately follows from the chain rule and we have $\mathbf{D}_1 = 0$.
In the case of $n=k >1$.
Put $D=\partial_{x_{i_2}} \cdots \partial_{x_{i_k}}$.
Fix $p \in U$ and let $p_t := p + te_{i_1}$, where $t\in \mathbb{R}$ and $e_{i_1}$ is the vector whose $i_1$-th component is $1$ and other ones are $0$.
Then, by induction hypothesis, there exists $\mathbf{D}_{k-1}: U \rightarrow \mathcal{D}_{k-2}(\mathbb{R}^d)$ such that for sufficiently small $t$,
\begin{align}
 D(h\circ F)(p_t) - \left[\bigg( \prod_{j=2}^k \partial_{x_{i_j}} F_{p_t} \cdot \bm{\partial}\bigg)h\right] (F(p_t)) + (\mathbf{D}_{k-1}(p_t)h)(F(p_t)).
\label{t to 0}
\end{align} 
We define a smooth function from $U$ to $\mathcal{D}_{k-1}(\mathbb{R}^d)$ by
\[
\mathbf{D}_k(p) := \left. \frac{{\rm d}}{{\rm d}t}
\left\{\mathbf{D}_{k-1}(p_t) - \prod_{j=1}^k \big( \partial_{x_{i_j}} F_{p_t} \cdot \bm{\partial} \big) \right\}  \right|_{t=0} 
+ \big(\partial_{x_{i_1}}F_{p_t} \cdot \bm{\partial}\big)\cdot \mathbf{D}_{k-1}(p).
\] 
Then, we apply ${\rm d}/{\rm d}t|_{t=0}$ to \eqref{t to 0}, and by direct computation, we have the formula \eqref{to prove formula dayo} in the case of $n=k$.
\end{proof}

We define the pull-back $f^*: \mathcal{E}(X) \rightarrow \mathcal{E}(X)$ by allocating $h \circ f$ to $h \in \mathcal{E}(X)$.
Then $f^*$ induces a continuous linear map on $\mathcal{E}(X)$ by Lemma \ref{lem: faa di bruno}.
The dual operator $f_*:=(f^*)': \mathcal{E}(X)' \rightarrow \mathcal{E}(X)'$ is also a continuous linear operator.
We note that $f^*\iota = \iota C_f$ holds.

Let $p \in X$ and fix a local coordinate $\phi$ from an open neighborhood of $p$ onto an open subset $\mathbb{R}^d$.
Then, we define an injective linear map $\delta_{p,\phi}\colon \mathcal{D}(\mathbb{R}^d) \rightarrow \mathcal{E}(X)'$ by
\begin{align}
\delta_{p,\phi}(D)(h):= D(h\circ \phi^{-1})(\phi(p))  \label{delta p phi}
\end{align}
for $D \in \mathcal{D}(\mathbb{R}^d)$ and $h \in \mathcal{E}(X)$.
By Lemma \ref{lem: faa di bruno}, the set $\delta_{p, \phi}(\mathcal{D}(\mathbb{R}^d))$ is independent of the choice of the local coordinate $\phi$, and we define subspaces of $\mathcal{E}(X)'$ by 
\begin{align*}
  \mathcal{D}(X)_p  & := \delta_{p,\phi}(\mathcal{D}(\mathbb{R}^d)), \\
  \mathcal{D}_n(X)_p & := \delta_{p,\phi}(\mathcal{D}_n(\mathbb{R}^d)). 
\end{align*}
We equip $\mathcal{E}(X)'$ with an ascending filtration $\{\mathcal{D}_n(X)_p\}_{n \ge 0}$.
We also provide an ascending filtration $\{\iota'(\mathcal{D}_n(X)_p)\}_{n\ge0}$ with $V$.
If $X$ is an open subset of $\mathbb{R}^d$, we always take the local coordinate $\phi$ as a canonical inclusion and denote by $\delta_p$ instead $\delta_{p, \phi}$.

As an immediate consequence of Lemma \ref{lem: faa di bruno}, we have the following statement:
\begin{lemma}
\label{lem: explicit push on a fixed point}
Let $p \in X$ be a fixed point of $f$, namely $f(p) = p$. 
Let $\phi$ be a diffeomorphism from an open neighborhood of $p$ into an open subset of $\mathbb{R}^d$.
Then, for any $\partial_{x_{i_1}} \cdots \partial_{x_{i_n}} \in \mathcal{D}_n(\mathbb{R}^d)$ ($i_j \in \{1,\dots,d\}$), we have
\begin{align}
  f_*[\delta_{p,\phi}(\partial_{x_{i_1}} \cdots\partial_{x_{i_n}})] - \delta_{p,\phi}\left( \prod_{j=1}^d \big(\partial_{x_{i_j}} f^\phi_{\phi(p)} \cdot \bm{\partial}\big)\right) \in \mathcal{D}_{n-1}(X)_p.
    \label{explicit push}
\end{align}
Here, $f^\phi := \phi \circ f \circ \phi^{-1}$ is defined on a sufficiently small neighborhood of $\phi(p)$, and the other notations are as in Lemma \ref{lem: faa di bruno}.
\end{lemma}
As a result, we have the following corollaries:
\begin{corollary} \label{cor: preserving property}
We use the same notation as in Lemma \ref{lem: explicit push on a fixed point}.
Then, we have 
\[f_*(\mathcal{D}_n(X)_p) \subset \mathcal{D}_n(X)_p.\]
Moreover, we have the following commutative diagram:
\[
\xymatrix@C=40pt{
  {\rm gr}[\mathcal{D}(X)_p]\ar[r]_{{\rm gr}_{f_*}}&{\rm gr}[\mathcal{D}(X)_p] \\
\mathbb{C}[t_1,\dots, t_d]\ar[r]^{S({\rm d}f^\phi_{\phi(p)})}\ar[u]^{\cong}&
\mathbb{C}[t_1,\dots, t_d]\ar[u]^{\cong}
}
\]
where the vertical maps are defined by the correspondence $t^\alpha \mapsto \delta_{p, \phi}(\partial_{x}^\alpha)$ and $S({\rm d}f^\phi_{\phi(p)})$ is the linear map defined by the correspondence $t_i \mapsto \sum_{m=1}^d \frac{\partial f^\phi_m}{\partial x_i}(\phi(p))t_m$, namely it satisfies
\[
S({\rm d}f^\phi_{\phi(p)})(t^\alpha) = \prod_{i=1}^d \left(\sum_{m=1}^d \frac{\partial f^\phi_m}{\partial x_i}(\phi(p))t_m \right)^{\alpha_i}.
\]
\end{corollary}
\begin{proof}
  By Lemma \ref{lem: explicit push on a fixed point}, $f_*(\mathcal{D}_n(X)_p) \subset \mathcal{D}_n(X)_p$ is obvious.
 By \eqref{explicit push}, ${\rm gr}_{f_*}$ allocates the element $\delta_{p,\phi}\big((\partial_{x}f^\phi(\phi(p)) \cdot \bm{\partial})^\alpha \big)$ to $\delta_{p,\phi}\big(\partial_{x}^\alpha\big)$, and thus we have the second statement.
\end{proof}

For $n \ge 0$, we define a surjective linear map induced by $\iota'$:
\begin{align}
  \kappa^n_{p}\colon \mathcal{D}_n(X)_p/\mathcal{D}_{n-1}(X)_p \longrightarrow \iota'\left(\mathcal{D}_n(X)_p\right)/\iota'\left(\mathcal{D}_{n-1}(X)_p\right), 
\end{align}
Here, we define $\mathcal{D}_{-1}(X)_p := \{0\}$.
Now, we provide the proof of Theorem \ref{mainthm}.
\begin{proof}[Proof of Theorem \ref{mainthm}]
  Let $\alpha$ be an eigenvalue of ${\rm d}f^r_p$.
  By Corollary \ref{cor: preserving property}, for any $n\ge 0$, there exists $v_n \in {\rm gr}^n[\mathcal{D}(X)_p] \setminus \{0\}$
  such that ${\rm gr}^n_{f_*^r}(v_n) = \alpha^n v_n$, and thus we have ${\rm gr}^n_{{C_f^r}'}(w_n) = \alpha^nw_n$, where $w_n := {\kappa}^n_{p}(v_n)$.
  Since $v_n \notin {\rm Ker} ( {\rm gr}^n_{f_*^r} )$, we see that $w_n \neq 0$ for infinitely many $n\ge0$. 
We note that $\|C_f'\|^r \ge \| {\rm gr}^n_{{C_f^r}'} \|$, where norms $\| \cdot \|$ for these two operators are the operator norms with respect to the quasi-norms on $V$ and its subquotient spaces, respectively. 
Thus, we have
  \begin{align}
    |\alpha| \le \|C_f'\|^{r/n}. \label{inequality 1}
  \end{align}
    Since \eqref{inequality 1} holds for infinitely many $n$, we have $|\alpha| \le 1$.
\end{proof}
This theorem implies that if a dynamical system induces a bounded composition operator on a certain functional quasi-Banach space, it does not have periodic points with a repelling direction.

\subsection{Boundedness of composition operators and stabilities of holomorphic dynamics}\label{subsec: stability cpx mfd}
Let $X$ be a complex manifold of complex dimension $d$.
We equip $\mathcal{A}(X)$ with the topology of uniform convergence on compact sets.
Here, we provide several corresponding results in the complex case.
Let 
\begin{align*}
 \mathcal{D}^{\rm hol}(\mathbb{C}^d) := \bigoplus_{n_1, \dots, n_d = 0}^\infty \mathbb{C} \partial_{z_1}^{n_1}\cdots \partial_{z_d}^{n_d}
\end{align*}
be the space of complex differential operators on $\mathcal{A}(\mathbb{C}^d)$ and for each $n \ge 0$, we define
\begin{align*}
  \mathcal{D}_n^{\rm hol}(\mathbb{C}^d) := \bigoplus_{n_1+ \dots + n_d \le n} \mathbb{C} \partial_{z_1}^{n_1}\cdots \partial_{z_d}^{n_d}.
\end{align*}
Let $p \in X$ and fix a local complex coordinate $\phi$ from an open neighborhood of $p$ onto an open subset $\mathbb{C}^d$.
Then, we define an injective linear map 
\[\delta_{p,\phi}^{\rm hol}\colon \mathcal{D}^{\rm hol}(\mathbb{C}^d) \longrightarrow \mathcal{A}(X)'\] 
in the same way as \eqref{delta p phi}. 
We also define $\mathcal{D}^{\rm hol}(X)_p$ and $\mathcal{D}_n^{\rm hol}(X)_p$ by the images of $\mathcal{D}^{\rm hol}(\mathbb{C}^d)$ and $\mathcal{D}_n^{\rm hol}(\mathbb{C}^d)$ under $\delta_{p, \phi}^{\rm hol}$, respectively.
In terms of Lemma \ref{lem: explicit push on a fixed point} and Corollary \ref{cor: preserving property}, for $n \ge 0$, we define
\begin{align*}
  {\rm gr}^{n, {\rm hol}}_{f_*}&\colon \mathcal{D}_n^{\rm hol}(X)_p/\mathcal{D}_{n-1}^{\rm hol}(X)_p \longrightarrow \mathcal{D}_n^{\rm hol}(X)_p/\mathcal{D}_{n-1}^{\rm hol}(X)_p,\\ 
  \kappa^{n, {\rm hol}}_p &\colon \mathcal{D}_n^{\rm hol}(X)_p/\mathcal{D}_{n-1}^{\rm hol}(X)_p \longrightarrow \iota'\left(\mathcal{D}_n^{\rm hol}(X)_p\right)/\iota'\left(\mathcal{D}_{n-1}^{\rm hol}(X)_p\right), 
\end{align*}
where $p$ is a fixed point of $f$. 

We also obtain a similar result of Theorem \ref{thm: stability of dynamics} with a similar proof:
\begin{theorem} \label{thm: stability of dynamics complex case}
Let $X$ be a complex manifold and $f\colon X \rightarrow X$ be a holomorphic map.
Let $V \subset \mathcal{A}(X)$ be a quasi-Banach space with continuous inclusion $\iota: V \rightarrow \mathcal{A}(X)$.
Let $p \in X$ be a periodic point of $f$ such that $f^r(p) = p$ for some $r>0$.
Assume the following two conditions:  
\begin{enumerate}
  \item $C_f$ is a bounded linear operator on $V$, and 
  \item\label{condition 2 complex} for infinitely many $n\ge 0$, ${\rm Ker}(\kappa^{n, {\rm hol}}_{p}) \subset {\rm Ker}({\rm gr}^{n, {\rm hol}}_{f_*^r})$.
\end{enumerate}
  Then, any eigenvalue $\alpha$ of ${\rm d}f^r_p: T_p^{1,0}(X) \rightarrow T_p^{1,0}(X)$ satisfies the inequality $|\alpha|\le 1$. Here, $T_p^{1,0}(X)$ is the $(1,0)$ part of the tangent vector over $\mathbb{C}$ of $X$ at $p$
\end{theorem}

We prove a simple but important lemma for the complex case:
\begin{lemma}
\label{lem: density and finite-dimensional criterion}
Let $X$ be a connected complex manifold. 
Let $V\subset \mathcal{A}(X)$ be a quasi-Banach space with a continuous inclusion $\iota\colon V \hookrightarrow \mathcal{A}(X)$.
Then, $V$ is infinite-dimensional if and only if 
\[\iota'(\mathcal{D}^{\rm hol}_n(X)_p) / \iota'(\mathcal{D}^{\rm hol}_{n-1}(X)_p) \neq 0\] 
for infinitely many $n\ge0$.
\end{lemma}
\begin{proof}
We prove the ``only if'' part as the ``if'' part is obvious.
Suppose there exists $k \ge 0$ such that for any $n\ge k+1$, 
\[\iota'(\mathcal{D}^{\rm hol}_n(X)_p) / \iota'(\mathcal{D}^{\rm hol}_{n-1}(X)_p) = 0.\] 
Then, we see that 
\begin{align}
\iota'(\mathcal{D}^{\rm hol}(X)_p) = \iota'(\mathcal{D}^{\rm hol}_k(X)_p). \label{eq: finite dimensionality}
\end{align}
We define a linear map by 
\[F: V \rightarrow \bigoplus_{\substack{\alpha \in \mathbb{Z}_{\ge 0}^d \\ |\alpha| \le k}} \mathbb{C};~ h \mapsto \big(h^{(\alpha)}(p)\big)_{|\alpha| \le k}, \]
where $h^{(\alpha)}(p) := \delta_{p, \phi}(\partial_z^\alpha)(h)$ with a local coordinate $\phi$ at $p$.
Then, by \eqref{eq: finite dimensionality}, we see that $F$ is injective as $X$ is connected.
Therefore, we conclude that $V$ is finite-dimensional.
\end{proof}
As a corollary, if $X$ is 1-dimensional, we have the following statement:
\begin{corollary} \label{cor: condition 2 riemann surface}
Assume $X$ is a connected complex manifold of 1 dimension.
Let $V \subset \mathcal{A}(X)$ be a quasi-Banach space with continuous inclusion $\iota\colon V \hookrightarrow \mathcal{A}(X)$.
Unless $V$ is finite dimensional, for infinitely many $n \ge 0$, we have ${\rm Ker}(\kappa^{n,\rm hol}_p) = \{0\}$, namely $\kappa^{n, \rm hol}_p$ is an isomorphism.
In particular, if $V$ is infinite-dimensional, the condition \eqref{condition 2} of Theorem \ref{thm: stability of dynamics complex case} is always true for any holomorphic map on $X$.
\end{corollary}
\begin{proof}
Since $\mathcal{D}_n^{\rm hol}(X)_p/ \mathcal{D}_{n-1}^{\rm hol}(X)_p$ is always one-dimensional and $\kappa^{n, \rm hol}_p$ is a surjective linear map onto $\iota'(\mathcal{D}^{\rm hol}_n(X)_p) / \iota'(\mathcal{D}^{\rm hol}_{n-1}(X)_p)$, the homomorphism $\kappa^{n, \rm hol}_p$ is isomorphism if and only if 
\[\iota'(\mathcal{D}^{\rm hol}_n(X)_p) / \iota'(\mathcal{D}^{\rm hol}_{n-1}(X)_p) \neq 0.\] 
Thus, the corollary follows from Lemma \ref{lem: density and finite-dimensional criterion}.
\end{proof}

\section{Function spaces continuously included in the space of smooth or holomorphic functions} \label{sec: function spaces}
Let $X$ be a smooth manifold of dimension $d$.
This section provides several important and typical examples of function spaces on $X$ continuously included in $\mathcal{E}(X)$.

\subsection{Reproducing kernel Hilbert space (RKHS)}
The theory of RKHS is among the most typical frameworks used to construct a Hilbert space continuously included in $\mathcal{E}(X)$ or $\mathcal{A}(X)$.
We first review a general notion of RKHS
\begin{definition}
Let $X$ be an arbitrary abstract nonempty set and let $k:X\times X\rightarrow\mathbb{C}$ be a map.
We say $k$ is a positive definite kernel if for any finite $p_1, \dots, p_n \in X$, the matrix $(k(p_i, p_j))_{i,j=1,\dots,n}$ is a Hermitian positive semi-definite matrix.
\end{definition}
Then, there is a classical result, known as the Moore-Aronszajn theorem, that states that for any positive definite  kernel $k$, there uniquely exists a Hilbert space $H_k$ composed of $\mathbb{C}$-valued maps on $X$ such that the following two conditions hold:
\begin{enumerate}
\item for any $p\in X$, the map $k_p:= k(p,\cdot)$ is an element of $H_k$, and
\item for any $p\in X$ and $h\in H_k$, we have $\langle h, k_p\rangle_{H_k}=h(p)$. \label{reproducing property}
\end{enumerate}
We call $H_k$ a {\em reproducing kernel Hilbert space}, or RKHS for short (see \cite{SS} for more details).
\begin{remark} \label{rmk: characterization of RKHS}
Let $H$ be an abstract Hilbert space and $H_k$ be an RKHS associated with a positive definite kernel $k$ on $X$.
If there exists a map $\varphi: X \to H$ such that $\langle \varphi(p), \varphi(p') \rangle_H = k(p,p')$ and a linear subspace generated by $\varphi(X)$ is dense in $H$, then the correspondence $\varphi(p) \mapsto k_p$ define a well-defined isomorphism from $H$ to $H_k$.
Thus, as a Hilbert space, RKHS is characterized as a pair of Hilbert space $H$ and map $\varphi:X\rightarrow H$ satisfying the above properties.
\end{remark}

In this paper, we say a positive definite kernel $k$ is {\em smooth} (resp. {\em holomorphic}) if $k_p \in \mathcal{E}(X)$ (resp. $\mathcal{A}(X)$) for any $p \in X$.
Note that by the symmetric property of $k$ if $k$ is a smooth (resp. holomorphic) positive definite kernel, then $k$ is smooth (resp. anti-holomorphic) with respect to the first variable as a function on $X \times X$.
As in the following proposition, $H_k$ inherits regularity of $k$:
\begin{proposition}
\label{prop: Hk subset CinftyX}
Let $X$ be a  {\em smooth} (resp. complex) manifold, and let $k:X \times X \rightarrow \mathbb{C}$ be a smooth (resp. holomorphic) positive definite kernel.
Then, $H_k$ is contained in $\mathcal{E}(X)$ (resp. $\mathcal{A}(X)$).
Moreover, the natural inclusion $H_k \hookrightarrow \mathcal{E}(X)$ (resp. $\mathcal{A}(X)$) is continuous.
\end{proposition}
\begin{proof}
It follows from \cite[Theorem 2.6, 2.7]{SS}.
\end{proof}
It should be noted that RKHS associated with a smooth (resp. holomorphic) positive definite kernel is characterized as a Hilbert space realized as a subspace of $\mathcal{E}(X)$ (resp. $\mathcal{A}(X)$), in other words, we have the following proposition:
\begin{proposition}
\label{prop: 1to1correspond smooth}
Let $X$ be a smooth (resp. complex) manifold.
Then, the correspondence $k \mapsto H_k$ provides a one-to-one correspondence between smooth (holomorphic) positive definite kernels on $X$ and Hilbert spaces continuously included in $\mathcal{E}(X)$ (resp. $\mathcal{A}(X)$).
\end{proposition}
\begin{proof}
We construct an inverse correspondence of $k \mapsto H_k$.
Let $H \subset \mathcal{E}(X)$ be a Hilbert space such that the inclusion $H \hookrightarrow \mathcal{E}(X)$ is continuous.
For $p \in X$, the linear functional $h \mapsto h(p)$ on $H$ is continuous.
Thus, by the Riesz representation theorem, there exists $k_p \in H$ such that $\langle h, k_p\rangle_H = h(p)$.
Then, $k(p,p') := \langle k_p, k_{p'}\rangle_H$ defines a positive definite kernel on $X$.
Since $H \subset \mathcal{E}(X)$, $k$ is obviously a smooth function on $X\times X$, and we have $H=H_k$ by the uniqueness of RKHS.
\end{proof}
Since $H_k$ is a Hilbert space, we may identify $H_k'$ with $H_k$ via the Riesz representation theorem (note that this identification is antilinear).
We provide an explicit formula of an element of $\iota'(\mathcal{D}(X)_p)$ as follows:
\begin{proposition} \label{prop: iota prime D}
Let $X$ be a smooth manifold and $H_k$ be an RKHS associated with a smooth positive definite kernel $k$ on $X$.
We identify $H_k'$ with $H_k$ via the Riesz representation theorem.
Let $\iota\colon H_k \hookrightarrow \mathcal{E}(X)$ be the continuous inclusion.
Then, for any $p\in X$, $v \in \mathcal{D}(X)_p$ and $p_0 \in X$, we have
\[\iota'(v)(p_0) = \overline{ v(k_{p_0}) }.\]
\end{proposition}
\begin{proof}
In fact, by definition, we have
\begin{align*}
\iota'(v)(p_0) 
  &= \langle \iota'(v), k_{p_0} \rangle_{H_k},\\
  &= \overline{v(k_{p_0})}.
\end{align*}
\end{proof}

\subsection{Positive definite kernels defined by a formal power series} \label{subsec: pdf fps}
Here, we consider a positive definite kernel defined by a formal power series.
It covers various typical positive definite kernels and provides a natural generalization of those in \cite{SS17, DKL17}.
In addition, in some particular cases, it coincides with a Bergman kernel on an open domain of $\mathbb{C}^d$, and thus provides quasi-Banach spaces.
We explain the Fock spaces treated in numerous works.

\subsubsection{RKHS in a general setting}
Let $\Phi(z)=\sum_{\alpha}c_\alpha z^\alpha $ be a holomorphic function and $\Omega \subset \mathbb{C}^d$ be an open neighborhood of $0$.  Assume that for $z \in \Omega$, the series 
\begin{align*}
  \Phi(|z_1|^2, \cdots, |z_d|^2) = \sum_{n \in \mathbb{Z}^d_{\ge 0}}^\infty c_\alpha |z|^{2\alpha},
\end{align*}
absolutely converges and
\begin{align}
  c_{\alpha} \ge 0 \text{ for all }\alpha \in \mathbb{Z}^{d}_{\ge 0}. \label{positive coefficient}
\end{align}
 Then, we can define a holomorphic positive definite kernel on $\Omega$ by 
\begin{align}
  k(z,w) := \Phi(\overline{z_1}w_1,\dots,\overline{z_d} w_d ). \label{kernel from fps}
\end{align}
\begin{remark}
  The condition \eqref{positive coefficient} is actually a necessary and sufficient condition for the positive definiteness of $k(z,w)$ defined as in \eqref{kernel from fps}, namely, 
  for any holomorphic function on $\Omega$, if $k(z,w) := \Phi(\overline{z_1}w_1,\dots,\overline{z_d} w_d )$ determines a positive definite kernel, then the coefficients $c_{n_1, \dots, n_d}$ are non-negative.
  In fact, let $\varepsilon >0$ be a sufficiently small positive number and ${\bf e}(\theta) := (\varepsilon {\rm e}^{in_1\theta}, \dots, \varepsilon{\rm e}^{in_d\theta})$.
  Since $k$ is positive definite, we see that 
  \[
    c_{n_1,\dots, n_d} 
    = \frac{1}{4\pi^2 \varepsilon^{2\sum_jn_j}}\int_0^{2\pi} \int_0^{2\pi} {\rm e}^{i \sum_j n_j(\theta-\phi)}k({\bf e}(\theta), {\bf e}(\phi)) {\rm d}\theta {\rm d}\phi 
    \ge 0.
  \]
\end{remark}
Let
\begin{align}
  \mathcal{I}_\Phi := \{n \in \mathbb{Z}^d_{\ge 0} : c_n \neq 0 \}.
\end{align}
Then, we consider the following assumption on $\mathcal{I}_\Phi$:
\begin{assumption} \label{ass: zariski dense condition}
The set $\mathcal{I}_\Phi$ is an infinite set, and for any nonemptyse proper subset $I \subsetneq \{1,\dots, d\}$ and element $m = (m_i)_{i \in I} \in \mathbb{Z}^{I}_{ \ge 0}$, there exist infinitely many $n = (n_i)_{i=1}^d \in \mathcal{I}_\Phi$ such that $n_i = m_i$ for $i \in I$.
\end{assumption}
  For example, $\Phi(z) = \phi(z_1+ \dots + z_d)$ for a transcendental entire function $\phi$ satisfies this assumption.
It provides a positive definite kernel treated in \cite{SS17, DKL17}.  
Bergman kernels on the ball or polydisc also satisfy this assumption.

Then, we have the following proposition:
\begin{proposition} \label{prop: condition 2 for fps kernels}
The notation is as above.
  Then, for any $p \in \Omega$ and $\alpha \in \mathbb{Z}^d_{\ge 0}$, we have
\begin{align*}
  \iota'\left( \delta_p^{\rm hol}\big( \partial_{z}^\alpha \big) \right)(w) 
  =  w^\alpha \big(\partial_{z}^\alpha \Phi\big)(\overline{p_1}w_1, \dots, \overline{p_d} w_d ).
\end{align*}
Here, we regard $H_k' = H_k$ via the Riesz representation theorem (via an antilinear isomorphism).
  Moreover, if the set $\mathcal{I}_\Phi$ satisfies Assumption \ref{ass: zariski dense condition}, then for any $p\neq 0$, $\iota'|_{\mathcal{D}(X)_p}$ is injective (and thus $\kappa_p^n$ is injective for $n\ge0$).
 In particular, the condition \eqref{condition 2} in Theorem \ref{thm: stability of dynamics complex case} holds for any holomorphic map $f$ on $\Omega$.
\end{proposition}
\begin{proof}
The first formula follows from Proposition \ref{prop: iota prime D}.
We prove the second statement.
Let $p \in \mathbb{C}^d \setminus \{0\}$.
Since for any $Q \in \mathbb{C}[t_1,\dots, t_d]$,
\[
\iota'(\delta_{p}^{\rm hol}(Q(\partial_z)) = \sum_{\alpha} \overline{b_\alpha} w^\alpha (\partial_z^\alpha \Phi)(\overline{p_1}w_1,\dots, \overline{p_d}w_d),
\] 
it suffices to show that for any $\Phi$ satisfying Assumption \ref{ass: zariski dense condition} and $Q(t) = \sum_{\alpha} b_\alpha t^\alpha \in \mathbb{C}[t_1,\dots, t_d]$, 
\begin{align}
 \sum_{\alpha} b_\alpha w^\alpha (\partial_z^\alpha \Phi)(\overline{p_1}w_1,\dots, \overline{p_d}w_d) = 0
\label{to prove 0426}
\end{align} 
implies $b_\alpha = 0$.
We prove it by induction of $d$.
In the case of $d=1$, \eqref{to prove 0426} is equivalent to
\[ c_m p_1^m \sum_{\alpha} b_\alpha m (m - 1) \cdots (m- \alpha+1) = 0\]
for all $m \ge \alpha$.
Since $p_1 \neq 0$ and $c_m \neq 0$ for infinitely many $m$ by Assumption \ref{ass: zariski dense condition}, we have $b_\alpha = 0$.
We next consider the case of $d>1$.
We only treat the case of $p_1 \neq 0$, and other cases are proved similarly.
Suppose $Q \neq 0$.
Then, there exists $Q_0 \in \mathbb{C}[t_1, \dots, t_d]$ such that $Q_0(t_1,\dots, t_{d-1}, 0) \neq 0$ and $Q = t_d^r Q_0$.
Let $Q_0 := \sum_\alpha b'_\alpha t^\alpha$.
Then, we have
\begin{align*}
w_d^r  \sum_{\alpha} b_\alpha' w^\alpha (\partial_z^\alpha (\partial_{z_d}^r \Phi))(\overline{p_1}w_1,\dots, \overline{p_d}w_d) 
& =  \sum_{\alpha} b_\alpha w^\alpha (\partial_z^\alpha \Phi)(\overline{p_1}w_1,\dots, \overline{p_d}w_d) \\
& = 0.
\end{align*}
Thus, we see that
\[\sum_{\alpha} b_\alpha' w^\alpha (\partial_z^\alpha (\partial_{z_d}^r\Phi))(\overline{p_1}w_1,\dots, \overline{p_d}w_d) = 0.\] 
Let $\Psi(z_1, \dots, z_{d-1}) := \partial_{z_d}^r\Phi(z_1,\dots, z_{d-1}, 0)$.
We note that $\Psi$ also satisfies Assumption \ref{ass: zariski dense condition}.
Then, by substituting $0$ for $w_d$ in the above formula, we have
\[\sum_{\substack{\alpha \\ \alpha_d = 0}} b_\alpha' w^\alpha (\partial_z^\alpha \Psi)(\overline{p_1}w_1,\dots, \overline{p_{d-1}}w_{d-1}) = 0.\] 
This equality corresponds to \eqref{to prove 0426} in the case $\Phi = \Psi$ and $Q=Q_0(t_1,\dots, t_{d-1},0)$.
Therefore, by induction hypothesis, $b'_\alpha = 0$ for any $\alpha = 0$ with $\alpha_d \neq 0$, which contradicts the fact that $Q_0(t_1,\dots,t_{d-1},0) \neq 0$.
\end{proof}

\subsubsection{Fock spaces} \label{subsubsec: fock space}
Let ${\rm A}$ be a Lebesgue measure on $\mathbb{C}^d$ via the isomorphism $ \mathbb{R}^d \times \mathbb{R}^d \cong \mathbb{C}^d;(x,y) \mapsto x+iy$.
For $\alpha > 0$, let 
\[\Phi(z) := \left( \frac{\alpha}{\pi} \right)^d {\rm e}^{\alpha z},\]
and define $k$ as in \eqref{kernel from fps}.
For $0 < q \le \infty$ and $h \in \mathcal{A}(\mathbb{C}^d)$, we define
\[ \| h \|_{q,\alpha} := 
\begin{cases}
  \displaystyle \int_{\mathbb{C}^{d}} |h(z)|^q {\rm e}^{-\alpha q |z|^2/2} {\rm dA}(z) <\infty &\text{ if }q<\infty, \\
  \displaystyle \sup_{z \in \mathbb{C}^d} |h(z)| {\rm e}^{-\alpha |z|^2/2} & \text{ if }q=\infty.
\end{cases}
\] 
Then, for $0 < q \le \infty$, we define
\[{F}_\alpha^q := \left\{ 
h \in \mathcal{A}(\mathbb{C}^d) :
\|h \|_{q, \alpha} < \infty
\right\}. \]
Then, for any $h \in {F}_\alpha^q$ and $p \in \mathbb{C}^d$, we have the following formula (see, for example, \cite[Theorem 8]{LW21}):
\[ h(p) = \int_{\mathbb{C}^d} h(z) \overline{k_p(z)} {\rm e}^{-\alpha |z|^2} {\rm dA}(z). \]
In particular,  ${F}_\alpha^q$ is a continuously included subspace of $\mathcal{A}(X)$ and for any $Q \in \mathbb{C}[t_1,\dots, t_d]$,
\begin{align}
Q(\partial_z)h(p) &= \int_{\mathbb{C}^d} h(z) \overline{\overline{Q}(\alpha z)k_p(z)} {\rm e}^{-\alpha |z|^2} {\rm dA}(z). \label{iotaprime for fock}
\end{align}
In particular, $\iota'(\delta_p(Q(\partial_z))) \in (F_\alpha^q)'$ is represented by the function $\overline{Q}(\alpha z)k_p(z)$ via the integral on the right hand side of \eqref{iotaprime for fock}, and we have the following proposition:
\begin{proposition} \label{prop: condition 2 for fock spaces}
  Let $\alpha > 0$ and $0< q \le \infty$.
 The linear map $\kappa_p^n$ is injective for the Fock space ${F}_\alpha^q$ for any $p \in  \mathbb{C}^d$ and $n \ge 0$
\end{proposition}
\begin{proof}
It suffices to show that $\iota'|_{\mathcal{D}^{\rm hol}(X)_p}$ is injective.
Since $\iota'(\delta_p(Q(\partial_z))) = 0 $ if and only if $\overline{Q}(\alpha z)k_p(z)=0$.
Since $k_p$ has zero nowhere, we see that $\overline{Q}(\alpha z)k_p(z) = 0$ if and only if $Q = 0$.
\end{proof}

\subsection{Shift-invariant kernels on the Euclidean spaces} \label{ex: shift invariant kernels Rd}
In this subsection, we introduce several RKHSs and Banach spaces defined by continuous positive definite functions, a Fourier transform of finite Borel measure on the Euclidean space.
We first provide a general example of RKHS and then give a Banach space where the base measure is absolutely continuous.

\subsubsection{RKHS for continuous positive definite functions}
Let $\mu$ be a Borel measure on $\mathbb{R}^d$ such that for any non-negative integer $n \ge 0$,
\[\int_{\mathbb{R}^d} |\xi|^n {\rm d}\mu(\xi) < \infty.\]
We employ $\xi = (\xi_1,\dots, \xi_d)$ to describe the variables of functions in $L^2(\mu)$.
We define a smooth positive definite kernel on $\mathbb{R}^d$ by 
\[k(x,y) := \widehat{\mu}(x-y) := \int {\rm e}^{i(x-y)\cdot \xi}{\rm d}\mu(\xi). \] 
Here for $x,y \in \mathbb{C}^d$, we denote ${}^txy$ by $x\cdot y$.

We note that the correspondence $k_p \mapsto {\rm e}^{i p\cdot \xi}$ induces an isomorphism (see Remark \ref{rmk: characterization of RKHS}). 
\[\rho: H_k \cong L^2(\mu); k_p \mapsto {\rm e}^{i p \cdot \xi}.\] 
Then, we have the following proposition:
\begin{proposition} \label{prop: condition 2 shift invariant}
The notation is as above, and we regard $H_k'$ as $H_k$ via the Riesz representation theorem (via an antilinear isomorphism).
  Then, for any $Q \in \mathbb{C}[t_1,\dots, t_d]$ and $p \in \mathbb{R}^d$, we have
  \begin{align}
\rho\iota'(\delta_p(Q(\partial_{x}))) = \overline{Q}(i\tilde{\xi}){\rm e}^{ip\cdot \xi},\label{prove sitaiyatu}
  \end{align}
where $\tilde{\xi}_j$ is the equivalence class of $\xi_j$ in $L^2(\mu)$.
In particular, if $\mathbb{C}[\xi_1,\dots, \xi_d] \hookrightarrow L^2(\mu)$, $\iota'|_{\mathcal{D}(X)_p}$ is injective for all $p \in \mathbb{R}^d$ (and $\kappa_p^n$ is injective for any $n\ge 0$), and the condition \eqref{condition 2} in Theorem \ref{thm: stability of dynamics} holds for any smooth function on $\mathbb{R}^d$.
\end{proposition}
\begin{proof}
We first remark that for any $h\in H_k$, $\rho(h)$ is characterized by  
\begin{align*}
 h(x) = \int_{\mathbb{R}^d} \rho(h)(\xi)e^{-ix\cdot\xi}{\rm d}\mu(\xi) \label{characterization of the isom}
\end{align*}
for all $ x \in \mathbb{R}^d$.
Since $k(x,y) = \int {\rm e}^{i(x-y)\cdot\xi} {\rm d}\mu(\xi)$, by Lemma \ref{prop: iota prime D}, we have
\[\iota'(\delta_p(Q(\partial_{x})))(x) = \int_{\mathbb{R}^d} \overline{Q}(i\xi) {\rm e}^{ip\cdot\xi} {\rm e}^{-ix\cdot\xi} {\rm d}\mu(\xi).\] 
We have \eqref{prove sitaiyatu} according to the first remark. 
\end{proof}
If $\mu$ has a Zariski dense support (for example, $\mu$ is absolutely continuous), the natural map $\mathbb{C}[\xi_1,\dots, \xi_d] \rightarrow L^2(\mu)$ is injective.

If we additionally impose $\int \prod_j {\rm e}^{a_j|\xi_j|} {\rm d}\mu(\xi) < \infty$ for some $a = (a_1,\dots, a_d) \in \mathbb{R}^d_{>0}$, $k$ also define a positive definite kernel on a complex manifold $\mathbb{I}_a := \{ z=(z_j)_{j=1}^d \in \mathbb{C}^d  : |{\rm Im}(z_i)|<a_i\}$ of complex dimension $d$, where we set $k(z,w) := \widehat{\mu}(z-\overline{w})$ for $z,w \in \mathbb{I}_a$.
Then, we have a similar proposition as follows:
\begin{proposition} \label{prop: condition 2 shift invariant complex ver}
  Assume $\int \prod_j {\rm e}^{a_j |\xi_j|} {\rm d}\mu(\xi) < \infty$ for some $a = (a_1,\dots, a_d) \in \mathbb{R}^d_{>0}$.
We regard $H_k'$ as $H_k$ via the Riesz representation theorem (via an antilinear isomorphism).
  Then, for any $Q \in \mathbb{C}[t_1,\dots, t_d]$ and $p \in \mathbb{I}_a$,
  \[\rho\iota'(\delta_p(Q(\partial_{z}))) = \overline{Q}(i\tilde{\xi}) {\rm e}^{i\overline{p}\cdot \xi}.\]
  where $\tilde{\xi} = (\tilde{\xi}_j)_j$ and each $\tilde{\xi}_j$ is the equivalence class of $\xi_j$ in $L^2(\mu)$.
  In particular, if $\mathbb{C}[\xi_1,\dots, \xi_d] \hookrightarrow L^2(\mu)$, $\iota'|_{\mathcal{D}^{\rm hol}(\mathbb{I}_a)_p}$ is injective for all $p \in \mathbb{I}_a$ (and $\kappa_p^n$ is injective for any $n\ge0$), and the condition \eqref{condition 2} in Theorem \ref{thm: stability of dynamics} holds for any holomorphic maps on $\mathbb{I}_a$.
\end{proposition}
There are several works regarding composition operators on such RKHSs \cite{CCG, CC93, IIS20, KD20, HHK13}.
\begin{remark}
The function space treated in \cite{KD20, HHK13} is isomorphic to the above RKHS with respect to a sum of atomic measures with the change of variable $z \mapsto iz$.
\end{remark}

\subsubsection{Banach spaces defined by positive definite functions}
\label{ex: shift-invariant kernels Rd}
Assume $\mu$ is absolutely continuous, namely $\mu = w(\xi) {\rm d}\xi$ for some $w \in L^1(\mathbb{R}^d)$.
We further assume $w \in L^q(\mathbb{R}^d)$ for all $q > 0$.
We define
\begin{align*}
\mathcal{B}_w^q := \left\{ h \in \mathcal{E}(\mathbb{R}^d) \cap \mathcal{S}' : \mathcal{F}{h} \in L^q(\mathbb{R}^d, w^{-1}) \right\},
\end{align*}
where $\mathcal{S}'$ is the space of tempered distributions on $\mathbb{R}^d$, $\mathcal{F}$ is the Fourier transform on $\mathcal{S}'$ such that $\mathcal{F}[h](x) = \int h(\xi){\rm e}^{-i\xi\cdot x} {\rm d}\xi$ for any Schwartz function $h$, and $L^q(\mathbb{R}^d, w^{-1})$ is the space composed of the measurable functions $h$ on $\mathbb{R}^d$ such that $h(\xi) = 0$ on $\{w(\xi) = 0 \}$ and $\int |h|^q w^{-1}(\xi){\rm d}\xi < \infty$.
Then, we have the following proposition \cite[Theorem 4.1]{FASSHAUER2015115}:

\begin{proposition}
Assume $\mu = w(\xi) {\rm d}\xi$ for some $w \in \cap_{r>0}L^r(\mathbb{R}^d)$.
Then, $\mathcal{B}_w^q$ is a Banach space isometrically isomorphic to $L^q(\mu)$ and the inclusion $\iota: \mathcal{B}_w^q \hookrightarrow \mathcal{E}(X)$ is continuous.
In addition, if we further assume $\int \prod_j{\rm e}^{a_j |\xi_j|} {\rm d}\mu(\xi) < \infty$ for some $a = (a_1,\dots, a_d) \in \mathbb{R}^d_{>0}$, we have $\mathcal{B}_w^q \subset \mathcal{A}(\mathbb{I}_a)$ and its inclusion map is continuous as well.
\end{proposition}

\begin{proof}
The first statement follows from the fact that the linear map $\theta: L^q(\mu) \rightarrow \mathcal{B}^q_w$ defined by
\[\theta(g) := \mathcal{F}^{-1}(gw^{2/q}) \] 
is actually an isometric isomorphism (see the proof of \cite[Theorem 4.1]{FASSHAUER2015115}).
The second and third statements follow from the Fourier inversion formula: for any $Q \in \mathbb{C}[t_1,\dots, t_d]$, $h \in \mathcal{B}_w^q$ and $p \in \mathbb{R}^d$, we have
\begin{align}
Q(\partial_x)h(p) &= \int_\mathbb{R^d} \theta^{-1}(h)(\xi) Q(-i\xi) w(\xi)^{2/q-1}{\rm e}^{-ip\cdot \xi} {\rm d}\mu(\xi) \label{explicit formula for sik}\\
 &\le C\|h\|_{\mathcal{B}^q_w} \nonumber
\end{align}
for some $C>0$ by the H\"older inequality.
\end{proof}
A measurable function $g$ on $\mathbb{R}^d$ with suitable decay can be embedded in $(\mathcal{B}^q_w)'$ when we regard $g$ as $h \mapsto \int \theta^{-1}(h)\overline{g}{\rm d}\mu$.
Thus, by \eqref{explicit formula for sik}, the $\iota'(\delta_p(Q(\partial_x))) \in (\mathcal{B}^q_w)'$ is described by the function $\overline{Q}(i\xi) w(\xi)^{2/q-1}{\rm e}^{ip\cdot \xi}$.
Then, we obtain statements similar to Propositions \ref{prop: condition 2 shift invariant} and \ref{prop: condition 2 shift invariant complex ver} hold.
In particular, we have the following proposition:

\begin{proposition}
Assume $\mu = w(\xi) {\rm d}\xi$ for some $w \in \cap_{r>0}L^r(\mathbb{R}^d)$.
Then, $\mathcal{B}_w^q$ satisfies the condition \eqref{condition 2} in Theorem \ref{thm: stability of dynamics} for any smooth map on $\mathbb{R}^d$.
If we further assume $\int \prod_j{\rm e}^{a_j |\xi_j|} {\rm d}\mu(\xi) < \infty$ for some $a = (a_1,\dots, a_d) \in \mathbb{R}^d_{>0}$, the space $\mathcal{B}_w^q$ satisfies the condition \eqref{condition 2} in Theorem \ref{thm: stability of dynamics complex case} for any holomorphic map on $\mathbb{I}_a$.
\end{proposition}

\section{Bounded composition operators for the complex affine space} \label{sec: Cd}
In this section, we discuss bounded composition operators on quasi-Banach spaces continuously included in $\mathcal{A}(\mathbb{C}^d)$.
At first, we treat the one-dimensional case and prove that maps except affine maps cannot induce bounded composition operators (Theorem \ref{thm: affineness in C, banach}).
Next, we discuss higher dimensional cases.
In the higher-dimensional case, the situation is much more complicated.
We prove that polynomial automorphisms, except affine maps, cannot induce bounded composition operators under mild conditions in the two-dimensional case (Theorem \ref{thm: two dim case}).
We conclude this paper with remarks for general dimensional cases.

\subsection{One-dimensional case}
Here, we provide the proof of Theorem \ref{mainthm2}, affineness of a holomorphic dynamical system which induces a bounded composition operator on a quasi-Banach subspace of $\mathcal{A}(\mathbb{C})$, by combining several classical results of holomorphic dynamics in one variable with the results proved in this paper.
und
\begin{proof}[Proof of Theorem \ref{mainthm2}]
Suppose that $f$ is not affine but that $C_f$ is a bounded linear operator on $V$.
By the theory of holomorphic dynamics of one variable, it is known that $f$ has (infinitely many) periodic points $p \in \mathbb{C}$ such that $|f'(p)|>1$ unless $f$ is an affine map (for example, see \cite[Theorem 1.20]{Sch10}). 
Such a point $p$ is called {\em repelling} periodic point.
Therefore, by Theorem \ref{thm: stability of dynamics complex case} and Corollary \ref{cor: condition 2 riemann surface}, we may assume $V$ is finite-dimensional.
We first assume there exists a nonconstant entire function $h \in V$ and $\lambda \in \mathbb{C}$ such that $C_f[h] = \lambda h$.
There exists a nonempty perfect set $J_f$ (so-called the Julia set) that is contained in the closure of the set of periodic points of $f$ (\cite[Theorem 4.2.7]{Bea91} and \cite[Theorem 1.20]{Sch10}).
Thus, there exists $p \in \mathbb{C}$ such that $f^n(p) = p$ and $h(p)\neq0$ for some $n\ge 1$.
Then, we have $h(p) = h(f^n(p))=\lambda^nh(p)$, and thus $\lambda^n =1$.
As $f^n$ is also not an affine map, we may assume $\lambda =1$, namely $h \circ f = h$.
Since the Julia set $J_f$ is the same as the closure of $\cup_{n\ge0} f^{-n}(S)$ for a finite set $S \subset \mathbb{C}$ of cardinality greater than $2$ (\cite[Theorem 4.2.7]{Bea91} and \cite[Theorem 1.7]{Sch10}), there exists $z \in \mathbb{C}$ such that $J_f$ is contains in the closure of $\cup_{n \ge 0} f^{-n}(z)$.
Thus, $h(J_f) \subset \{h(z)\}$, and we conclude that $h$ is constant, which is contradiction.
Next, we assume that there is no eigenfunction of $C_f$ other than constant functions.
Then, a nonconstant function exists $h \in V$ such as $C_f[h] = h+1$.
Let $p \in \mathbb{C}$ be a periodic point of $f$.
Let $r$ be the period, namely  $f^r(p) = p$.
Then, we see that $h(p) = C_f^r[h](p) = h(p) + r$, and thus $r=0$.
We again have a contradiction.
Therefore, $f(z)$ is $az + b$ for some $a, b \in \mathbb{C}$.
\end{proof}

\subsection{Two-dimensional cases}
Here, we discuss a higher dimensional case, especially the two-dimensional case, under the condition that $\kappa_p^{n, \rm hol}$ is injective for all but finitely many $p \in \mathbb{C}^d$ and infinitely many $n\ge0$.
In contrast to the one-dimensional case, the relationship between the behavior of dynamics and the boundedness of composition operators gets much more complicated.
Let $V \subset \mathcal{A}(\mathbb{C}^d)$ be a quasi-Banach space, and its inclusion be continuous.
We define
\begin{align}
\mathcal{G}_d(V) := \left\{ A \in {\rm GL}_d(\mathbb{C}) : C_{A(\cdot)+b}\text{ is bounded on $V$ for some $b\in\mathbb{C}^d$}\right\}.
\end{align}
We note that $\mathcal{G}_d(V)$ is a sub semigroup of ${\rm GL}_d(\mathbb{C})$.

We first consider a two-dimensional case.
We call a holomorphic map $f=(f_1, f_2): \mathbb{C}^2 \rightarrow \mathbb{C}^2$ algebraic if $f_1$ and $f_2$ are polynomials.
If there exists another algebraic map $g: \mathbb{C}^2 \rightarrow \mathbb{C}^2$ such that $f\circ g = g \circ f = {\rm id}$, we call $f$ a polynomial automorphism.
In the two-dimensional case, an elementary transform $e_{Q,a,b,c}$ and a H\'enon transform $h_{Q,b}$ are important examples (and they are essentially only examples of non-trivial polynomial automorphisms on $\mathbb{C}^2$ by the structure theorem of the automorphism group \cite{friedland_milnor_1989}):
\begin{align*}
e_{P,a,b,c}(x,y) &:= (ax+P(y), by + c),\\
h_{Q,b}(x,y) &:= (Q(x) - by, x),
\end{align*}
where $a,b,c \in \mathbb{C}$ with $a, b \neq 0$, $P$ is a polynomial, and $Q$ is a polynomial of degree greater than $1$.

Now, we prove Theorem \ref{mainthm3}.
\begin{proof}[Proof of Theorem \ref{mainthm3}]
Let $f$ be a non-affine polynomial automorphism on $\mathbb{C}^2$.
As in \cite{friedland_milnor_1989}, we may describe $f = g_n \circ \dots \circ g_1$ with a reduced word $(g_1,\dots, g_n)$ (see \cite[Definition, p.69]{friedland_milnor_1989}), where we may assume each $g_i$ is either a non-upper triangular regular matrix or elementary transform.

If $n$ is even, then $f$ is conjugate to a cyclically reduced element (\cite[p.70]{friedland_milnor_1989}).
Thus, by \cite[Theorem 2.6]{friedland_milnor_1989}, the map $f$ is a finite composition of generalized H\'enon transforms $h_{Q_1,b_1}\circ \dots \circ h_{Q_r, b_r}$ for some $r>0$.
Then, by \cite[Theorem 3.4]{BS92}, there exists a periodic point $p \in \mathbb{C}^2$ of period $r$ such that the absolute value of an eigenvalue of ${\rm d}f^r_p$ is greater than 1.
Thus, by Theorem \ref{thm: stability of dynamics complex case}, $f$ cannot induce a bounded composition operator on $V$.

Suppose $n$ is odd.
If $g_n$ is a elementary transform, by Lemma \ref{lem: nazo lemma} below, there exists $A \in \mathcal{G}_2(V)$ such that $A=(a_{ij})$ with $a_{21} \neq 0$ and $h := A(\cdot) + b$ induces a bounded composition operator for some $b \in \mathbb{C}$.
Then, $h\circ f$ is conjugate to a finite composition of generalized H\'enon transforms by \cite[Theorem 2.6]{friedland_milnor_1989}. 
Assume that $g_n$ is a regular matrix.
If $g_1g_n$ is not an upper triangular matrix, then $g_n^{-1} f g_n$ is cyclically reduced, and thus as in the above argument, $C_f$ cannot become a bounded linear operator.
We now assume $g_1g_n$ is an upper triangular matrix.
We may assume $g_n$ is in the form of
\[
g_n =
\left(\begin{array}{cc} 1 & -\alpha \\ 0 & 1  \end{array} \right) 
\left(\begin{array}{cc} \beta & 0\\ 0 & \gamma  \end{array} \right)
\left(\begin{array}{cc} 0 & 1\\ 1 & 0  \end{array} \right)
.
\] 
Since
\[
\left(\begin{array}{cc} 1 & \alpha \\ 0 & 1  \end{array} \right)
\left(\begin{array}{cc} a_{11}& a_{12} \\ a_{21} & a_{22} \end{array} \right)
\left(\begin{array}{cc} 1 & -\alpha \\ 0 & 1  \end{array} \right)
=
\left(\begin{array}{cc} * & -a_{21}\alpha^2 - \alpha(a_{11}-a_{22}) + s_{12} \\ * & *  \end{array} \right),\]
there exists $A \in \mathcal{G}_2(V)$ such that $g_n^{-1} A g_n$ is not an upper triangular matrix by Lemma \ref{lem: nazo lemma} below.
Let $h := A(\cdot) + b$ such that $C_h$ is bounded.
Then, we see that 
\[ g_n^{-1}\circ h\circ f \circ g_n = (g_n^{-1} \circ h \circ g_n )\circ g_{n-1}\circ \cdots \circ (g_2\circ g_1g_n)\] 
is conjugate to a finite composition of generalized H\'enon transforms, by \cite[Theorem 2.6]{friedland_milnor_1989} again.
Therefore, we conclude that there exists an affine map $h$ such that $C_h$ is bounded and $h \circ f$ is conjugate to a finite composition of generalized H\'enon transforms for any $n$.
If $f$ induces bounded composition operators, then $h \circ f$ also does the same, and thus $C_f$ cannot be bounded.
Therefore, $f$ must be an affine map if $C_f$ is bounded on $V$. 

\end{proof}
This theorem implies that if sufficiently many affine maps induce bounded composition operators on the function spaces introduced in Section \ref{sec: function spaces}, no polynomial automorphism except affine maps can induce a bounded composition operator.

We expect that the higher dimensional version of Theorem \ref{thm: two dim case} is valid not only for polynomial automorphisms but also for any holomorphic maps.
In fact, it is actually proved in specific quasi-Banach spaces.
For example, Carswell, MacCluer, and Schuster proved the result in \cite{CMS}, while Cho, Choe, and Koo established the same for the Fock-Sobolev space in \cite{HBH14}.
When the kernel is in the form of $k(z,w) = \Phi({}^t\overline{z}w)$ for an entire function $\Phi$, which is a special form of the kernel introduced in Section \ref{subsec: pdf fps}, Luan, Khoi, and Le proved it in \cite{DKL17}, and Stochel and Stochel also solved this problem in \cite{SS17}.
In our previous work \cite{IIS20}  with Ikeda and Sawano, we investigated the RKHS associated with a shift-invariant kernel (Section \ref{ex: shift-invariant kernels Rd}).
We treated certain general class of shift invariant kernels, and the condition $\langle \mathcal{G}_d(V) \rangle_{\mathbb{C}} = {\rm M}_d(\mathbb{C})$ appears as well.
However, we did not use any results developed in the theory of holomorphic dynamics, and introduced this condition in a different context.
We expect that our results (\cite[Theorem 1]{IIS20}) hold without the technical condition (called Assumption (B) there). 

The following lemma is used in the above proof of Theorem \ref{thm: two dim case}.
\begin{lemma}
\label{lem: nazo lemma}
Let $\mathcal{G} \subset {\rm GL}_2(\mathbb{C})$ be a sub semigroup.
For $S=(s_{ij}) \in \mathcal{G}$, let 
\[ \mathfrak{A}_S :=   \left\{\alpha \in \mathbb{C} : s_{21}\alpha^2+(s_{11}-s_{22})\alpha -s_{12} = 0 \right\}. \] 
Then $\langle \mathcal{G} \rangle_\mathbb{C} = {\rm M}_2(\mathbb{C})$ if and only if $\cap_{S \in \mathcal{G}} \mathfrak{A}_S = \emptyset$ and there exists $B = (b_{ij}) \in \mathcal{G}$ such that $b_{21} \neq 0$.
\end{lemma}
\begin{proof}
We note that for $S = (s_{ij})$, $\alpha$ satisfies the equation $s_{21}\alpha^2+(s_{11}-s_{22})\alpha -s_{12} = 0$ if and only if there exists $\lambda \neq 0$ such that $(1,\alpha)(S-\lambda) = 0$.
In fact, since 
\begin{align*}
s_{21}\alpha^2+(s_{11}-s_{22})\alpha -s_{12} 
& = \det \left( \begin{array}{cc}s_{11}+s_{21}\alpha & s_{12} + s_{22}\alpha \\ 1 & \alpha \end{array}\right) \\
& =0,
\end{align*}
there exists $\lambda \in \mathbb{C}$ such that 
\[
\lambda (1,\alpha) = (1,\alpha) 
\left( \begin{array}{cc}s_{11} & s_{12} \\ s_{21}& s_{22}\end{array}\right).
\] 
We here regard any element in $\mathbb{C}^2$ as a horizontal vector.
We first prove the ``only if'' part.
Assume $\cap_{S\in \mathcal{G}}\mathfrak{A}_S \neq \emptyset$.
Let $\alpha \in \cap_{S\in \mathcal{G}}\mathfrak{A}_S$.
Then, for any $B \in \langle \mathcal{G} \rangle_\mathbb{C}$, 
\[(1,\alpha) B = \lambda_B (1,\alpha)\] 
for some $\lambda_B \in \mathbb{C}$, thus $\left( \begin{array}{cc} 0 & 1 \\ 0 & 0\end{array} \right)$ is not contained in $\langle \mathcal{G} \rangle_\mathbb{C}$.
We next prove the ``if'' part.
Assume $\cap_{S\in \mathcal{G}}\mathfrak{A}_S =\emptyset$.
Fix $B = (b_{ij}) \in \mathcal{G}$ such that $b_{21} \neq 0$.
Since $\mathfrak{A}_B \neq \emptyset$, take $\alpha \in \mathfrak{A}_B$ and let $v := (1,\alpha)$.
Let $\lambda \in \mathbb{C}$ such that $vB=\lambda v$.    %
We take another element $C \in \mathcal{G}$ such that $C \notin \mathbb{C} + \mathbb{C}B$.
It suffices to show that there exists $D \in \mathcal{G}$ such that $D \notin \mathbb{C} + \mathbb{C}B + \mathbb{C}C$.
If $BC \notin \mathbb{C} + \mathbb{C}B + \mathbb{C}C$, the element $BC$ is the desired one.
Suppose $BC = a + bB + cC$ for some $a,b,c \in \mathbb{C}$.
Then, we have
\begin{align*}
vBC = \lambda v C = av + b\lambda v + cvC.
\end{align*}
Thus, we have
\[
(\lambda - c)vC = (a + b \lambda) v.
\] 
In the case of $\lambda \neq c$, since $vC = (\lambda - c)^{-1}(a + b\lambda) v$, we have $\alpha \in \mathfrak{A}_C$, and thus $\alpha \in \mathfrak{A}_B \cap \mathfrak{A}_C$.
In the case of $\lambda = c$.
Then, we see that $a = -b\lambda$, and we have 
\[(\lambda - B)(b-C)=0.\] 
Let $w = b_{21}^{-1}(0,1)(\lambda - B) = (1, \beta)$.
Then $w(b-C)=0$, namely $\beta \in \mathfrak{A}_C$ and  we have either $w(\lambda -B) \in \mathbb{C}w$ or $C=b$.
Thus, we see that $\beta$ or $\alpha$ must be contained in $\mathfrak{A}_B \cap \mathfrak{A}_C$.
Therefore, we have $\mathfrak{A}_B \cap \mathfrak{A}_C \neq \emptyset$.
Since $\mathfrak{A}_{B}\cap \mathfrak{A}_{C}\cap \mathfrak{A}_{C'} \neq \emptyset$ for any $C' \in \mathbb{C} + \mathbb{C}B + \mathbb{C}C$, there exists
$D \in \mathcal{G}$ such that $D \notin \mathbb{C} + \mathbb{C}B + \mathbb{C}C$ as we are now assuming $\cap_{S \in \mathcal{G}} \mathfrak{A}_S = \emptyset$.
\end{proof}

\section*{Acknowledgement} 
We would like to thank Professor Yutaka Ishii  giving us valuable commnets and a detailed explanation regarding holomorphic dynamics on $\mathbb{C}^2$ and higher-dimensional affine spaces.
This work was supported by a JST CREST JPMJCR1913, Japan, and JST ACTX JPMJAX2004, Japan.

\bibliographystyle{plain}   
\bibliography{reference}

\end{document}